\documentclass[11pt]{amsart}

\usepackage{amsmath,amsthm,amssymb,mathrsfs,amsfonts,verbatim,enumitem,color,leftidx}
\usepackage{mathabx}
\usepackage{etoolbox} 
\usepackage{bbm}
\usepackage[all,tips]{xy}
\usepackage{graphicx,ifpdf,tikz}
\usepackage{mathtools}
\usepackage{here}
\ifpdf
\fi
\usepackage{hyperref}
\hypersetup{
  colorlinks   = true, 
  urlcolor     = blue, 
  linkcolor    = blue, 
  citecolor   = orange 
}
\usepackage[right,displaymath,mathlines]{lineno}

\theoremstyle{definition}
\newtheorem{thm}{Theorem}[section]

\newtheorem{cor}[thm]{Corollary}

\newtheorem{de}[thm]{Definition}
\newtheorem{prop}[thm]{Proposition}

\theoremstyle{remark}
\newtheorem{rem}[thm]{Remark}

\numberwithin{equation}{section}

\makeatletter

\newcommand{\Rmnum}[1]{\expandafter\@slowromancap\romannumeral #1@}
\makeatother

\newcommand{\bv}{\mathbf{v}}

\title[Fundamental domain]{A fundamental domain for $PGL(2,\mathbb{F}_q[t])\backslash PGL\!\left(2,\mathbb{F}_q\!\left(\!(t^{-1})\!\right)\!\right)$}

\begin{document}

\author{Sanghoon Kwon}
\address{}
\email{skwon@cku.ac.kr}

\date{January 20, 2019}
\maketitle

\begin{abstract}
We give a strong fundamental domain for the quotient of $PGL_2\!\left(\mathbb{F}_q(\!(t^{-1})\!)\!\right)$ by $PGL_2\!\left(\mathbb{F}_q[t]\right)$ as a subset of distinct ordered triple points of $\mathbb{P}^1(\mathbb{F}_q(\!(t^{-1})\!))$.
\end{abstract}

\section{Introduction}\label{sec:1}

Given a topological space and a group acting on it, the images of a single point under the group action form an orbit of the action. A \emph{strong fundamental domain} is a subset of the space which contains exactly one point from each of these orbits.

The action of the modular group $\textrm{PSL}(2,\mathbb{Z})$ on the unit tangent bundle $\textrm{T}^1\mathbb{H}$ of the upper half plane $\mathbb{H}$ by Mobius transformation 
$$\left[\left(\begin{array}{cc} a & b \\ c & d\end{array}\right)\right]\colon (z,\bv)\mapsto \left(\frac{az+b}{cz+d},\frac{\bv}{(cz+d)^2}\right)\textrm{ for } z\in \mathbb{H},\bv\in \textrm{T}_z\mathbb{H}$$
serves as a key example of arithmetic and geometry. For each $(z,\bv)$ in $\textrm{T}^1\mathbb{H}$, we can find a neighborhood of $(z,\bv)$ which does not contain any other element of the $\textrm{PSL}(2,\mathbb{Z})$-orbit of $(z,\bv)$. 

There are various ways of constructing a strong fundamental domain, but a common choice is the union 
\begin{align*}\{(z,\bv)\colon z\in R,\bv\in \textrm{T}_z^1\mathbb{H}\}&\cup\left\{(w_1,\bv)\in \textrm{T}^1\mathbb{H}\colon 0\le\textrm{arg}(\bv)<\frac{\pi}{3}\right\}\\&\cup\left\{(w_2,\bv)\in\textrm{T}^1\mathbb{H}\colon 0\le\textrm{arg}(\bv)<\pi\right\}\end{align*} for two ramified points $w_1=\frac{1}{2}+\frac{\sqrt{3}i}{2}$ and $w_2=i$ and the region 
$$R = \left\{ z \in\mathbb{H} \colon | z |  > 1 ,  -\frac{1}{2}\le  \textrm{Re} ( z )     < \frac{1}{2}  \right\}\cup\left\{z\in\mathbb{H}\colon|z|=1,-\frac{1}{2}<\textrm{Re}(z)< 0\right\}$$   
bounded by the vertical lines $\textrm{Re}(z) =-\frac{1}{2}$ and $\textrm{Re}(z) =\frac{1}{2}$ and the circle $|z|=1$.

The boundary at infinity $\partial_\infty\mathbb{H}^2$ may be identified with $S^1$ and hence with $\mathbb{P}^1(\mathbb{R})$. Let us say that a mutually distinct ordered triple points $(x_1,x_2,x_3)\in\mathbb{P}^1(\mathbb{R})^3$ is \emph{positively ordered} if one reaches $x_2$ before $x_3$ when starting counterclockwise from $x_1$. We note that there is a bijection between the unit tangent bundle $\textrm{T}^1\mathbb{H}$ and the set of mutually distinct positively ordered triple points $(x_1,x_2,x_3)$ of $\mathbb{P}^1(\mathbb{R})$.

In this article, we construct a  fundamental domain for the action of modular group on projective general linear group over a field of formal series, namely, the action of $PGL_2(\mathbb{F}_q[t])$ on $PGL\!\left(2,\mathbb{F}_q\!\left(\!(t^{-1})\!\right)\right)$. 

Let $\mathbf{K}$ be the field $\mathbb{F}_q(\!(t^{-1})\!)$ of Laurent series in $t^{-1}$ over a finite field $\mathbb{F}_q$ and $\mathbf{Z}$ be the subring $\mathbb{F}_q[t]$ of polynomials in $t$ over $\mathbb{F}_q$ of $\mathbf{K}$. We further denote by $\mathcal{O}$ the local ring $\mathbb{F}_q[\![t^{-1}]\!]$ of $\mathbf{K}$ which consists of power series in $t^{-1}$ over $\mathbb{F}_q$. From now on, let $G=PGL_2\!\left(\mathbb{F}_q(\!(t^{-1})\!)\right)$, $\Gamma=PGL_2\!\left(\mathbb{F}_q[t]\right)$, $h$ the diagonal element $$\left[\left(\begin{array}{cc} t & 0 \\ 0 & 1\end{array}\right)\right]$$ of $G$ and $W=PGL_2(\mathcal{O})$ a maximal compact subgroup of $G$.  


Let us denote by $\mathbb{P}^1(\mathbf{K})^3_{\textrm{dist}}$ the set of \emph{mutually distinct} ordered triple points $(\omega_1,\omega_2,\omega_3)$ of $\mathbb{P}^1(\mathbf{K})$. Since two by two projective general linear group over any field $F$ acts simply transitively on distinct ordered triple points in $\mathbb{P}^1(F)$ by M\"{o}bius transformation, we have a bijection $\Phi\colon G\to\mathbb{P}^1(\mathbf{K})^3_{\textrm{dist}}$ given by $$\Phi(g)=g\cdot(0,1,\infty).$$

We state our main theorem.

\begin{thm}\label{thm:main} Given $(\omega_1,\omega_2,\omega_3)\in\mathbb{P}^1(\mathbf{K})^{3}_{\textrm{dist}}$, there is a unique $\gamma\in \Gamma$ such that $$\gamma\cdot(\omega_1,\omega_2,\omega_3)\in S_0\cap S_1\cap(S_2\cup S_3)$$ where \begin{align*}
S_0&=\{(\omega_1,\omega_2,\omega_3)\in\mathbb{P}^1(\mathbf{K})^{3}_{\textrm{dist}}\colon\textrm{the leading coefficient of }\omega_2\textrm{ is }1\} \\
S_1&=\{(\omega_1,\omega_2,\omega_3)\in\mathbb{P}^1(\mathbf{K})^{3}_{\textrm{dist}}\colon\deg\omega_1<0<\deg\omega_3\} \\
S_2&=\{(\omega_1,\omega_2,\omega_3)\in\mathbb{P}^1(\mathbf{K})^{3}_{\textrm{dist}}\colon\deg\omega_1\ne\deg\omega_2<\deg\omega_3\} \\
S_3&=\{(\omega_1,\omega_2,\omega_3)\in\mathbb{P}^1(\mathbf{K})^{3}_{\textrm{dist}}\colon\deg\omega_1=\deg\omega_2=\deg(\omega_1-\omega_2)\}.
\end{align*}
\end{thm}

In Section~\ref{sec:2}, we review the Bruhat-Tits tree of the group $PGL\!\left(2,\mathbb{F}_q\!\left(\!(t^{-1})\!\right)\right)$ with an explicit description of vertices and boundary at infinity. We prove that the set $S_0\cap S_1\cap(S_2\cup S_3)$ in Theorem~\ref{thm:main} contains at least one point of each $\Gamma$-orbit of $\mathbb{P}^1(\mathbf{K})^3_{\textrm{dist}}$ in Section~\ref{sec:3}. In Section~\ref{sec:4}, we complete the proof of the main theorem and further discuss the diagonal action on $\Gamma\backslash G$.

\section{Tree of $G$}\label{sec:2}
In this section, we review the Bruhat-Tits tree $\mathcal{T}$ of $G$. We give a concrete interpretation of vertices and edges of $\mathcal{T}$ without using Bruhat-Tits theory of general algebraic groups (see \cite{BT}).

Given an element $\displaystyle\alpha=\sum_{i=n}^{-\infty}a_it^i$ of $\mathbf{K}$ with $a_n\ne 0$, let us define
\begin{align*}
[\alpha]&=a_0+a_1t+\cdots+a_nt^n\\
\{\alpha\}&=a_{-1}t^{-1}+a_{-2}t^{-2}+\cdots \\
[\alpha]_L&=a_n \\
\deg \alpha&= n
\end{align*}
the polynomial part, fractional part, the leading term and the degree of $\alpha$, respectively.
Let $\mathcal{T}$ be the graph whose vertices are the elements of $G/W$, which we can describe as
$$\left[\left(\begin{array}{cc} t^n & f(t) \\ 0 & 1\end{array}\right)\right]W$$ for some integer $n$ and a rational function $f(t)\in t^{n+1}\mathbf{Z}$. Let $$\pi_n\colon t^{n}\mathbf{Z}\to t^{n+1}\mathbf{Z}$$ be the projection map which forgets the $t^n$ term. Two vertices 
$$\left[\left(\begin{array}{cc} t^{n_1} & f_1(t) \\ 0 & 1\end{array}\right)\right]W\textrm{ and }\left[\left(\begin{array}{cc} t^{n_2} & f_2(t) \\ 0 & 1\end{array}\right)\right]W$$ are adjacent to each other if and only if $|n_1-n_2|=1$ and $f_1$ and $f_2$ satisfy 
\[
 \left\{\begin{array}{lr}
          \displaystyle f_2(t)=\pi_{n_2}(f_1(t)), & \textrm{if }n_2=n_1+1\\
        f_2(t)=f_1(t)+ct^{n_1}, & \textrm{if }n_2=n_1-1 \\
        \end{array}\right.
  \]
for some $c\in \mathbb{F}_q$. It follows that the degree of every vertex of $\mathcal{T}$ is equal to $q+1$. Let us denote by $o$ the standard vertex 
$$\left[\left(\begin{array}{cc} 1 & 0 \\ 0 & 1\end{array}\right)\right]W.$$

\begin{de}\label{def:1} An isometry $r\colon \mathbb{Z}_{\ge 0}\to\mathcal{T}$ is called a \emph{parametrized geodesic ray} and an isometry $\ell\colon \mathbb{Z}\to\mathcal{T}$ is called a \emph{parametrized bi-infinite geodesic}. Two geodesic rays $\ell$ and $\ell'$ are said to be \emph{equivalent} if and only if $\{d_{\mathcal{T}}(\ell(n),\ell'(n))\colon n\in\mathbb{Z}_{>0}\}$ is bounded above. 
The \emph{Gromov boundary at infinity} of $\mathcal{T}$ is defined as the set of equivalence classes $[\ell]$ of geodesic ray $\ell$ starting from a fixed vertex $v$ of $\mathcal{T}$.
\end{de}
Then, the {Gromov boundary $\partial_\infty\!\mathcal{T}$ at infinity} of $\mathcal{T}$ can be identified with $\mathbf{K}\cup\{\infty\}$ (See Chapter 2 of \cite{Se}). Moreover, given any boundary point $\omega\in\partial_\infty\mathcal{T}$, there is a unique geodesic ray from the standard vertex $o$ to $\omega$. We also note that given any two distinct points $\omega_1$ and $\omega_2$ of $\partial_\infty\mathcal{T}$, there is a bi-infinite geodesic, which we will denote by $(\omega_1\omega_2)$, begins at $\omega_1$ and ends at $\omega_2$. 


Let $\partial_\infty\!\mathcal{T}^{3}_{\textrm{dist}}$ be the set $$\{(\omega_1,\omega_2,\omega_3)\in(\partial_\infty\!\mathcal{T})^3\colon \omega_i\ne\omega_j\textrm{ for }1\le i\ne j\le 3\}$$ of distinct ordered triple points in $\partial_\infty\!\mathcal{T}$. Since two by two projective general linear group over a field $F$ acts simply transitively on $(\mathbb{P}^1(F))^3_\textrm{dist}$ by M\"{o}bius transformation, we have a bijection $\Phi\colon G\to \partial_\infty\!\mathcal{T}^{3}_{\textrm{dist}}\simeq\mathbb{P}^1(\mathbf{K})^3_{\textrm{dist}}$ (see also \cite{PS}) given by $$\Phi(g)=g\cdot(0,1,\infty).$$ 

\begin{figure}[H]
\begin{center}
\begin{tikzpicture}[every loop/.style={}]
  \tikzstyle{every node}=[inner sep=0pt]
  \node (-1) at (-0.4,0) {$\cdots$};
  \node (0) {$\bullet$} node [above=4pt] at (0,0) {$x_{-3}$};
  \node (2) at (1.5,0) {$\bullet$} node [above=4pt] at (1.5,0) {$x_{-2}$}; 
  \node (4) at (3,0) {$\bullet$}node [above=4pt] at (3,0) {$x_{-1}$}; 
  \node (6) at (4.5,0) {$\bullet$}node [above=4pt] at (4.5,0) {$x_0$}; 
  \node (8) at (6,0) {$\bullet$}node [above=4pt] at (6,0) {$x_1$}; 
  \node (10) at (7.5,0) {$\bullet$}node [above=4pt] at (7.5,0) {$x_2$}; 
  \node (11) at (9,0) {$\bullet$}node [above=4pt] at (9,0) {$x_3$}; 
  \node (13) at (9.4,0) {$\cdots$};
  \node (14) at (-0.4,-2) {$\omega_1$};
  \node (15) at (3,-2) {$\omega_2$};
  \node (16) at (9.4,-2) {$\omega_3$};
  \node (17) at (-0.3,-1.55) {};
  \node (18) at (3,-1.5) {};
  \node (19) at (9.3,-1.55) {};
  \node (20) at (6,-2.5) {$x_i=\left[\left(\begin{array}{cc} t^i & 0 \\ 0 & 1 \end{array}\right)\right]W$};
  \draw[dashed] (-0.4cm,-2cm) circle (0.5cm);
  \draw[dashed] (3cm,-2cm) circle (0.5cm);
  \draw[dashed] (9.4cm,-2cm) circle (0.5cm);

  \path[-] (0) edge node [above=4pt] {} (2)
 (2) edge node [above=4pt] {} (4)
 (4) edge node [above=4pt] {} (6)
 (6) edge node [above=4pt] {} (8)
 (8) edge node [above=4pt] {} (10)
 (10) edge node [above=4pt] {} (11)
 (0) edge node {} (17)
 (4) edge node {} (18)
 (0) edge node {} (0,-1)
 (0) edge node {} (0.2,-1)
 (2) edge node {} (1.3,-1)
 (2) edge node {} (1.5,-1)
 (2) edge node {} (1.7,-1)
 (4) edge node {} (2.8,-1)
 (4) edge node {} (3.2,-1)
 (6) edge node {} (4.3,-1)
 (6) edge node {} (4.5,-1)
 (6) edge node {} (4.7,-1)
 (8) edge node {} (5.8,-1)
 (8) edge node {} (6,-1)
 (8) edge node {} (6.2,-1)
 (10) edge node {} (7.3,-1)
 (10) edge node {} (7.5,-1)
 (10) edge node {} (7.7,-1)
 (11) edge node {} (8.8,-1)
 (11) edge node {} (9,-1)
 (11) edge node {} (19);
\end{tikzpicture}
\caption{$\deg(\omega_1)=-3$, $\deg(\omega_2)=-1$ and $\deg(\omega_3)=3$}
\label{Fig:ex1}
\end{center}
\end{figure}
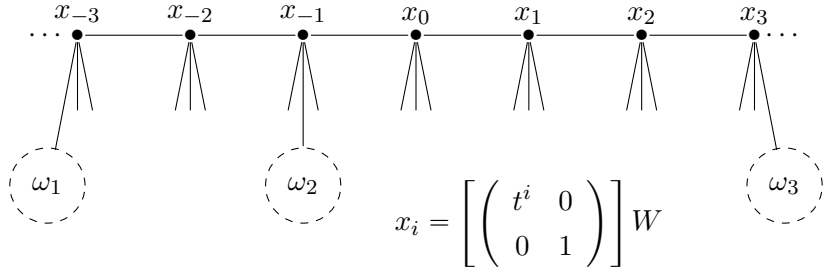




\section{Regular continued fraction and discrete geodesic flow}\label{sec:3}

In this section, we prove that the set of Theorem~\ref{thm:main} contains at least one point from each $\Gamma$-orbit. The main ingredients of the proof are the geometry of $\mathcal{T}$ and regular continued fraction expansion of elements in $\mathbf{K}$. 
 
Every formal series $\alpha$ can be unqiuely written as 
$$\displaystyle \alpha=a_0+\dfrac{1}{a_1+\dfrac{1}{a_2+\dfrac{1}{\ddots}}}$$
for $a_0\in\mathbb{F}_q[t]$ and non-constant polynomials $a_i$ for $i\ge 1$,
 which we call a \emph{regular continued fraction} of $\alpha$. The $a_i$ are called the \emph{partial quotients} of $\alpha$ and we will write $\alpha=[a_0;a_1,a_2,\ldots]$.
 
\begin{prop}\label{prop:subset} Given $(\omega_1,\omega_2,\omega_3)\in\partial_\infty\!\mathcal{T}^{3}_{\textrm{dist}}$, there is a $\gamma\in \Gamma$ such that $$\gamma\cdot(\omega_1,\omega_2,\omega_3)\in S_0\cap S_1\cap(S_2\cup S_3)$$ where \begin{align*}
S_0&=\{(\omega_1,\omega_2,\omega_3)\in\partial_\infty\!\mathcal{T}^{3}_{\textrm{dist}}\colon[\omega_2]_L=1\} \\
S_1&=\{(\omega_1,\omega_2,\omega_3)\in\partial_\infty\!\mathcal{T}^{3}_{\textrm{dist}}\colon\deg\omega_1<0<\deg\omega_3\} \\
S_2&=\{(\omega_1,\omega_2,\omega_3)\in\partial_\infty\!\mathcal{T}^{3}_{\textrm{dist}}\colon\deg\omega_1\ne\deg\omega_2<\deg\omega_3\} \\
S_3&=\{(\omega_1,\omega_2,\omega_3)\in\partial_\infty\!\mathcal{T}^{3}_{\textrm{dist}}\colon\deg\omega_1=\deg\omega_2=\deg(\omega_1-\omega_2)\}.
\end{align*}
\end{prop}
\begin{proof} Let \begin{align*}\omega_1&=[a_0;a_1,a_2,a_3,\ldots],\\ 
\omega_2&=[b_0;b_1,b_2,b_3,\ldots],\quad (a_i,b_i,c_i\in\mathbf{Z}\cup\{\infty\})\\
\omega_3&=[c_0;c_1,c_2,c_3,\ldots]
\end{align*} 
be the regular continued fraction expansions of $\omega_i$, $i=1,2,3$, respectively. Note that $$\iota=\left[\left(\begin{array}{cc} 0 & 1 \\ 1 & 0 \end{array}\right)\right],\sigma_c=\left[\left(\begin{array}{cc} c & 0 \\ 0 & 1 \end{array}\right)\right]\textrm{ and }u_f=\left[\left(\begin{array}{cc} 1 & f(t) \\ 0 & 1\end{array}\right)\right]$$ belong to $\Gamma$ for every $d\in\mathbb{F}_q$ and $f(t)\in \mathbf{Z}$. The actions are given by $$\iota(\omega_1,\omega_2,\omega_3)=\left(\frac{1}{\omega_1},\frac{1}{\omega_2},\frac{1}{\omega_3}\right)$$ $$\sigma_c(\omega_1,\omega_2,\omega_3)=(c\omega_1,c\omega_2,c\omega_3)$$ and $$u_f(\omega_1,\omega_2,\omega_3)=(\omega_1+f,\omega_2+f,\omega_3+f).$$ 
Given any $$(\omega_1,\omega_2,\omega_3)\in  S_1\cap (S_2\cup S_3),$$ we have $\sigma_c(\omega_1,\omega_2,\omega_3)\in S_0\cap S_1\cap (S_2\cup S_3)$ for $c=([\omega_2]_L)^{-1}\in \mathbb{F}_q$. Thus, it is enough to find a $\gamma\in \Gamma$ such that $\gamma\cdot(\omega_1,\omega_2,\omega_3)\in S_1\cap (S_2\cup S_3)$.

Since $\omega_1\ne\omega_3$, there exists $j\ge 0$ such that $a_j\ne c_j$ and hence we may assume without loss of generality that $a_0=0$ and $c_0\ne 0$ by applying $\iota$ and $u_{-a_0},\ldots,u_{-a_{j}}$. Let 
\begin{align*}\omega_1&=[0;a_1,a_2,a_3,\ldots],\\ 
\omega_2&=[b_0;b_1,b_2,b_3,\ldots],\quad (a_i,b_i,c_i\in\mathbf{Z}\cup\{\infty\})\\
\omega_3&=[c_0;c_1,c_2,c_3,\ldots]
\end{align*} 
If $b_0\ne 0$ and $\deg b_0<\deg c_0$, then $(\omega_1,\omega_2,\omega_3)\in S_1\cap S_2$. If $\deg b_0> \deg c_0$, then 
\begin{align*}
\iota\circ u_{-c_0}(\omega_1)&=[0;-c_0,a_1,a_2,\ldots] \\
\iota\circ u_{-c_0}(\omega_2)&=[0;b_0-c_0,b_1,b_2,\ldots] \\
\iota\circ u_{-c_0}(\omega_3)&=[c_1;c_2,c_3\ldots].
\end{align*}
In this case, $\iota\circ u_{-c_0}(\omega_1,\omega_2,\omega_3)\in S_1\cap S_2$. If $\deg b_0=\deg c_0$ and $b_0\ne c_0$, then $\iota\circ u_{-c_0}(\omega_1,\omega_2,\omega_3)$ belongs to $S_1\cap S_2$ or $S_1\cap S_3$ depending on whether $[b_0]_L=[c_0]_L$ or not. If $b_0=c_0$ and $k>0$ is the smallest integer such that $b_k\ne c_k$, then we have 
\begin{align*}
\iota\circ u_{-c_k}\circ\cdots\iota\circ u_{-c_0}(\omega_1)&=[0;-c_k,\ldots,-c_0,a_1,a_2,\ldots] \\
\iota\circ u_{-c_k}\circ\cdots\iota\circ u_{-c_0}(\omega_2)&=[0;b_k-c_k,b_{k+1},b_{k+2},\ldots] \\
\iota\circ u_{-c_k}\circ\cdots\iota\circ u_{-c_0}(\omega_3)&=[c_{k+1};c_{k+2},c_{k+3},\ldots]
\end{align*}
and hence it can be reduced to one of the previous cases. Thus, this proves the existence of such $\gamma$.
\end{proof}

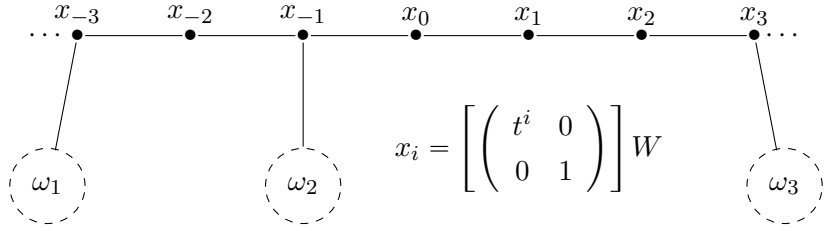
\begin{figure}[H]
\begin{center}
\begin{tikzpicture}[every loop/.style={}]
  \tikzstyle{every node}=[inner sep=0pt]
  \node (-1) at (-0.4,0) {$\cdots$};
  \node (0) {$\bullet$} node [above=4pt] at (0,0) {$x_{-3}$};
  \node (2) at (1.5,0) {$\bullet$} node [above=4pt] at (1.5,0) {$x_{-2}$}; 
  \node (4) at (3,0) {$\bullet$}node [above=4pt] at (3,0) {$x_{-1}$}; 
  \node (6) at (4.5,0) {$\bullet$}node [above=4pt] at (4.5,0) {$x_0$}; 
  \node (8) at (6,0) {$\bullet$}node [above=4pt] at (6,0) {$x_1$}; 
  \node (10) at (7.5,0) {$\bullet$}node [above=4pt] at (7.5,0) {$x_2$}; 
  \node (11) at (9,0) {$\bullet$}node [above=4pt] at (9,0) {$x_3$}; 
  \node (13) at (9.4,0) {$\cdots$};
  \node (14) at (-0.4,-2) {$\omega_1$};
  \node (15) at (3,-2) {$\omega_2$};
  \node (16) at (9.4,-2) {$\omega_3$};
  \node (17) at (-0.3,-1.55) {};
  \node (18) at (3,-1.5) {};
  \node (19) at (9.3,-1.55) {};
  \node (20) at (6,-1.5) {$x_i=\left[\left(\begin{array}{cc} t^i & 0 \\ 0 & 1 \end{array}\right)\right]W$};
  \draw[dashed] (-0.4cm,-2cm) circle (0.5cm);
  \draw[dashed] (3cm,-2cm) circle (0.5cm);
  \draw[dashed] (9.4cm,-2cm) circle (0.5cm);

  \path[-] (0) edge node [above=4pt] {} (2)
 (2) edge node [above=4pt] {} (4)
 (4) edge node [above=4pt] {} (6)
 (6) edge node [above=4pt] {} (8)
 (8) edge node [above=4pt] {} (10)
 (10) edge node [above=4pt] {} (11)
 (0) edge node {} (17)
 (4) edge node {} (18)
 (11) edge node {} (19);
\end{tikzpicture}
\caption{$(\omega_1,\omega_2,\omega_3)\in S_0\cap S_1\cap S_2$}
\label{Fig:ex2}
\end{center}
\end{figure}

\vspace{-1em}

\begin{figure}[H]
\begin{center}
\begin{tikzpicture}[every loop/.style={}]
  \tikzstyle{every node}=[inner sep=0pt]
  \node (-1) at (-0.4,0) {$\cdots$};
  \node (0) {$\bullet$} node [above=4pt] at (0,0) {$x_{-3}$};
  \node (2) at (1.5,0) {$\bullet$} node [above=4pt] at (1.5,0) {$x_{-2}$}; 
  \node (4) at (3,0) {$\bullet$}node [above=4pt] at (3,0) {$x_{-1}$}; 
  \node (6) at (4.5,0) {$\bullet$}node [above=4pt] at (4.5,0) {$x_0$}; 
  \node (8) at (6,0) {$\bullet$}node [above=4pt] at (6,0) {$x_1$}; 
  \node (10) at (7.5,0) {$\bullet$}node [above=4pt] at (7.5,0) {$x_2$}; 
  \node (11) at (9,0) {$\bullet$}node [above=4pt] at (9,0) {$x_3$}; 
  \node (13) at (9.4,0) {$\cdots$};
  \node (14) at (-0.6,-2) {$\omega_1$};
  \node (15) at (0.6,-2) {$\omega_2$};
  \node (16) at (9.4,-2) {$\omega_3$};
  \node (17) at (-0.3,-1.55) {};
  \node (18) at (0.3,-1.55) {};
  \node (19) at (9.3,-1.55) {};
  \node (20) at (6,-1.5) {$x_i=\left[\left(\begin{array}{cc} t^i & 0 \\ 0 & 1 \end{array}\right)\right]W$};
  \draw[dashed] (-0.6cm,-2cm) circle (0.5cm);
  \draw[dashed] (0.6cm,-2cm) circle (0.5cm);
  \draw[dashed] (9.4cm,-2cm) circle (0.5cm);

  \path[-] (0) edge node [above=4pt] {} (2)
 (2) edge node [above=4pt] {} (4)
 (4) edge node [above=4pt] {} (6)
 (6) edge node [above=4pt] {} (8)
 (8) edge node [above=4pt] {} (10)
 (10) edge node [above=4pt] {} (11)
 (0) edge node {} (17)
 (0) edge node {} (18)
 (11) edge node {} (19);
\end{tikzpicture}
\caption{$(\omega_1,\omega_2,\omega_3)\in S_0\cap S_1\cap S_3$}
\label{Fig:ex3}
\end{center}
\end{figure}
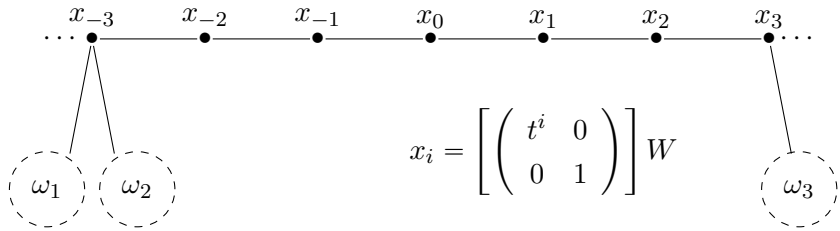
Now let us introduce some notions for geometry of trees. This is useful when we construct a strong fundamental domain of $\Gamma\backslash\partial_\infty\!\mathcal{T}^{3}_{\textrm{dist}}$ as well as a factor $(\Gamma\backslash\mathcal{GT},\phi)$ of $(\Gamma\backslash G,\phi_a)$ which is similar to the geodesic flow system of the unit tangent bundle on hyperbolic surfaces. Here, $\phi_h\colon \Gamma\backslash G\to \Gamma\backslash G$ is the right translation map given by $\phi_h(x)=xh$.

Recall that a parametrized bi-infinite geodesic in $\mathcal{T}$ is an isometry $\ell\colon\mathbb{Z}\to\mathcal{T}$. (See the below of Definition~\ref{def:1}). Note that a bi-infinite geodesic is completely determined by two distinct points $\ell(-\infty)$ and $\ell(\infty)$ of $\partial_\infty\!\mathcal{T}$ and the marked vertex $\ell(0)\in(\ell(-\infty)\ell(\infty))$.

 Let $\mathcal{GT}$ be the set of all parametrized bi-infinite geodesics in $\mathcal{T}$. Given a tuple $(\omega_1,\omega_2,\omega_3)\in\partial_\infty\!\mathcal{T}^{3}_{\textrm{dist}}$, we have the map $p_{(\omega_1\omega_3)}(\omega_2)$ be the unique vertex projected from $\omega_2$ to the (unparametrized) bi-infinite geodesic $(\omega_1\omega_3)$. For example, $p_{(\omega_1\omega_3)}(\omega_2)=x_{-1}$ for Figure~\ref{Fig:ex2} and  $p_{(\omega_1\omega_3)}(\omega_2)=x_{-3}$ for Figure~\ref{Fig:ex3}. 

We define a factor map $\Theta\colon\partial_\infty\!\mathcal{T}^{3}_{\textrm{dist}}\to\mathcal{GT}$ by $\Theta(\omega_1,\omega_2,\omega_3)=\ell$ where $\ell$ is the unique parametrized bi-infinite geodesic satisfying $\ell(-\infty)=\omega_1$, $\ell(\infty)=\omega_3$, and $\ell(0)=p_{(\omega_1\omega_3)}(\omega_2)$. Let $\phi\colon\mathcal{GT}\to\mathcal{GT}$ be the discrete geodesic flow given by $\phi(\ell)(n)=\ell(n+1)$. Taking the quotient by the discrete subgroup $\Gamma$ and attaching the factor map $\Theta$, we obtain the following commutative diagram.

\begin{equation*}
\begin{aligned}
\xymatrix{ \Gamma\backslash G \ar[r]^{\,\, \phi_h\quad} \ar[d]_{\Phi} & \Gamma\backslash G \ar[d]_{\Phi}   \\ 
\Gamma\backslash \partial_\infty\!\mathcal{T}^{3}_{\textrm{dist}} \ar[r]^{\,\, \phi_h\quad} \ar[d]_{\Theta} & \Gamma\backslash \partial_\infty\!\mathcal{T}^{3}_{\textrm{dist}} \ar[d]_{\Theta}   \\ 
\Gamma\backslash \mathcal{GT} \ar[r]^{\,\, \phi \quad} & \Gamma\backslash \mathcal{GT} }
\end{aligned}
\end{equation*}

\section{A fundamental domain for $\Gamma\backslash G$}\label{sec:4}
 In this section, we give a strong fundamental domain for the action of $\Gamma$ on $\mathbb{P}^1(\mathbf{K})^{3}_{\textrm{dist}}$. We recall that $G$ may be identified with $\mathbb{P}^1(\mathbf{K})^{3}_{\textrm{dist}}$ via the map $\Phi$ in Section~\ref{sec:1}.

\begin{thm}\label{thm:fundom} The subset $S_0\cap S_1\cap(S_2\cup S_3)$ of $\partial_\infty\!\mathcal{T}^3_{\textrm{dist}}$ introduced in Proposition~\ref{prop:subset} is a strong fundamental domain of $\Gamma\backslash \partial_\infty\!\mathcal{T}^{3}_{\textrm{dist}}$.
\end{thm}
\begin{proof}

Let us assume that $\gamma\cdot(\omega_1,\omega_2,\omega_3)=(\eta_1,\eta_2,\eta_3)$ and 
$$(\omega_1,\omega_2,\omega_3),(\eta_1,\eta_2,\eta_3)\in S_0\cap S_1\cap (S_2\cup S_3).$$
Then, $\gamma$ maps $p_{(\omega_1\omega_3)}(\omega_2)$ to $p_{(\eta_1\eta_3)}(\eta_2)$. Since $(\omega_1,\omega_2,\omega_3)$ and $(\eta_1,\eta_2,\eta_3)$ are contained in $S_1\cap(S_2\cup S_3)$, it follows that $p_{(\omega_1\omega_3)}(\omega_2)$ and $p_{(\eta_1\eta_3)}(\eta_2)$ are of the form
$$\left[\left(\begin{array}{cc} t^{\deg{\omega_2}} & 0 \\ 0 & 1 \end{array}\right)\right]W\quad\textrm{and}\quad\left[\left(\begin{array}{cc} t^{\deg{\eta_2}} & 0 \\ 0 & 1 \end{array}\right)\right]W,$$ respectively. Since $G$ is the disjoint union $$G=\underset{i=0}{\overset{\infty}{\sqcup}} \Gamma \left[\left(\begin{array}{cc} t^i & 0 \\ 0 & 1 \end{array}\right)\right]W$$ of double cosets with respect to $\Gamma$ and $W$, this implies that 
$\deg(\omega_2)=\pm\deg(\eta_2)$.

If $\deg(\omega_2)=\deg(\eta_2)\ge 0$, then $$\gamma=\left[\left(\begin{array}{cc} a & b \\ 0 & d \end{array}\right)\right]$$ for some $a,d\in\mathbb{F}_q^\times$ and $b\in \mathbf{Z}$ with $\deg(b)\le \deg(\omega_2)$. Since $$\deg(\eta_1)=\deg(\gamma\cdot\omega_1)=\deg(ad^{-1}\omega_1+d^{-1}b)<0,$$ we must have $b=0$. Moreover, $$[\omega_2]_L=[\eta_2]_L=[ad^{-1}\omega_2]_L=1$$ implies that 
$$\gamma=\left[\left(\begin{array}{cc} 1 & 0 \\ 0 & 1 \end{array}\right)\right].$$

If $\deg(\omega_2)=\deg(\eta_2)<0$, then $$\gamma=\left[\left(\begin{array}{cc} a & 0 \\ c & d \end{array}\right)\right]$$ for some $a,d\in\mathbb{F}_q^\times$ and $c\in\mathbf{Z}$ with $\deg(c)\le-\deg(\omega_2)$. The positivity of $\deg(\eta_3)$ and the similar argument implies that $\gamma$ is the identity.

If $\deg(\omega_2)=-\deg(\eta_2)$, then either $\gamma$ or $\gamma^{-1}$ is of the form
$$\left[\left(\begin{array}{cc}0 & d \\ a & b \end{array}\right)\right]$$ for some $a,d\in\mathbb{F}_q^\times$ and $\deg(b)\le|\deg(\omega_2)|$. Then $\deg(\eta_3)=\deg(d(a\omega_3+b)^{-1})<0$ which is impossible. Therefore, $\gamma$ must be an identity. Together with Proposition~\ref{prop:subset}, we can conclude that $S_0\cap S_1\cap(S_2\cup S_3)$ is a fundamental domain.
\end{proof}

Recall that $\phi_h$ is the right translation map $x\mapsto xh$ from $\Gamma\backslash G$ to itself. We discuss the system $(\Gamma\backslash G,\phi_h)$ in the remaining part. Let $\tau_h\colon G\to G $ be the right translation map $\tau_h(g)=gh$. 
\begin{prop}\label{prop:commuting}
Let $\varphi_h\colon \partial_\infty\!\mathcal{T}^{3}_{\textrm{dist}}\to \partial_\infty\!\mathcal{T}^{3}_{\textrm{dist}}$ be the map given by $$\varphi_h(\omega_1,\omega_2,\omega_3)=\left(\omega_1,\frac{(\omega_2-\omega_1)\omega_3t+(\omega_3-\omega_2)\omega_1}{(\omega_2-\omega_1)t+\omega_3-\omega_2},\omega_3\right),$$ then the following commutative diagram holds.
\begin{equation*}
\begin{aligned}
\xymatrix{ G \ar[r]^{\,\, \tau_h\quad} \ar[d]_{\Phi} & G \ar[d]_{\Phi}   \\ 
\partial_\infty\!\mathcal{T}^{3}_{\textrm{dist}} \ar[r]^{\,\, \varphi_h \quad} & \partial_\infty\!\mathcal{T}^{3}_{\textrm{dist}} }
\end{aligned}
\end{equation*}
\end{prop}

\begin{proof}

Suppose that $(\omega_1,\omega_2,\omega_3)=g\cdot (0,1,\infty)$ with $g=\left(\begin{array}{cc} a & b \\ c & d\end{array}\right)$.
Then, 
\begin{align*}
&\Phi\circ\phi_h\circ\Phi^{-1}(\omega_1,\omega_2,\omega_3)=\Phi\circ\phi_h(g)=\Phi(gh)=g\cdot(0,t,\infty)=(\omega_1,\omega_2',\omega_3)
\end{align*}
where 
$$\omega_2'=\frac{at+b}{ct+d}\quad\textrm{ and }\quad\omega_2=\frac{a+b}{c+d}.$$
If $c=0$, then $b=\omega_1d$, $a=(\omega_2-\omega_1)d$ and $\omega_3=\infty$. Thus, we have
$$\omega_2'=(\omega_2-\omega_1)t+\omega_1.$$
If $c\ne 0$, then $a=\omega_3c$ and $b=\omega_1d$. In this case, we have 
$$\frac{d}{c}=\frac{\omega_3-\omega_2}{\omega_2-\omega_1}$$ which yields
$$\omega_2'=\frac{(\omega_2-\omega_1)\omega_3t+(\omega_3-\omega_2)\omega_1}{(\omega_2-\omega_1)t+\omega_3-\omega_2}.$$ This completes the proof of the proposition.
\end{proof}

According to the Theorem~\ref{thm:fundom}, there is a unique point in $$S_0\cap S_1\cap(S_2\cup S_3)\cap\Gamma\cdot(\omega_2,\omega_2,\omega_3)$$ which we will denote by $[(\omega_1,\omega_2,\omega_3)]_\Gamma$.

\begin{cor}Let $S=S_0\cap S_1\cap(S_2\cup S_3)$ and $\psi_h\colon S\to S$ be the map given by
$$\psi_h(\omega_1,\omega_2,\omega_3)=\left[\left(\omega_1,\frac{(\omega_2-\omega_1)\omega_3t+(\omega_3-\omega_2)\omega_1}{(\omega_2-\omega_1)t+\omega_3-\omega_2},\omega_3\right)\right]_\Gamma.$$
Then, the diagram \begin{equation*}
\begin{aligned}
\xymatrix{ \Gamma\backslash G \ar[r]^{\,\, \phi_h\quad} \ar[d]_{\Phi} & \Gamma\backslash G \ar[d]_{\Phi}   \\ 
S \ar[r]^{\,\, \psi_h \quad} & S }
\end{aligned}
\end{equation*}
commutes.
\end{cor}

\end{document}